\documentclass[11pt]{article}
\usepackage{amsfonts}
\usepackage{latexsym}
\usepackage{cite}
\usepackage{amsmath,amsfonts,latexsym,amssymb,color}
\usepackage[mathscr]{eucal}
\usepackage{cases}
\usepackage{amsthm}

\usepackage[bf,small]{caption2}
\usepackage{float}
\usepackage{graphicx}
\usepackage{amsmath}
\usepackage{amssymb}
\usepackage[all]{xy}

\newtheorem{theorem}{theorem}[section]
\newtheorem{thm}[theorem]{Theorem}
\newtheorem{lem}[theorem]{Lemma}

\newtheorem{prop}[theorem]{Proposition}

\newtheorem{rmk}[theorem]{Remark}
\newtheorem{nota}[theorem]{Notation}

\begin{document}

\title{\textbf{Regular balanced Cayley maps on ${\rm PSL}(2,p)$}}
\author{\Large Haimiao Chen
\footnote{Email: \emph{chenhm@math.pku.edu.cn.\ This work is supported by NSFC-11401014.}} \\
\normalsize \em{Beijing Technology and Business University, Beijing, China}}
\date{}
\maketitle

\begin{abstract}
  A {\it regular balanced Cayley map} (RBCM for short) on a finite group $\Gamma$ is an embedding of a Cayley graph on $\Gamma$ into a surface, with some special symmetric property. People have classified RBCM's for cyclic, dihedral, generalized quaternion, dicyclic, and semi-dihedral groups. In this paper we classify RBCM's on the group ${\rm PSL}(2,p)$ for each prime number $p>3$.
\end{abstract}

\section{Introduction}

Let $\Gamma$ be a finite group and let $\Omega$ be a generating set not containing the identity and $\omega^{-1}\in\Omega$ whenever $\omega\in\Omega$.
The {\it Cayley graph}
${\rm Cay}(\Gamma,\Omega)$ is the graph having vertex set $V=\Gamma$ and arc set $\Gamma\times\Omega$, where $(\eta,\omega)$ means the arc from the vertex $\eta$ to $\eta\omega$. If $\rho$ is a cyclic permutation on $\Omega$, then it gives a cyclic order to the set of arcs starting from $\eta$ for each $\eta$, via $(\eta,\omega)\mapsto(\eta,\rho(\omega))$.
This determines a unique cellular embedding of the Cayley graph ${\rm Cay}(\Gamma,\Omega)$ into a closed oriented surface, where ``cellular" means that each connected component of the complement of the embedded graph is homeomorphic to a disk. Such an embedding is called a {\it Cayley map} and is denoted by $\mathcal{CM}(\Gamma,\Omega,\rho)$.

An \emph{isomorphism} between two Cayley maps is an isomorphism of the underlying graphs which is compatible with cyclic orders.

A Cayley map $\mathcal{CM}(\Gamma,\Omega,\rho)$ is called \emph{regular} if its automorphism group acts transitively on the arc set, and called \emph{balanced} if $\rho(\omega^{-1})=\rho(\omega)^{-1}$ for all $\omega\in \Omega$. From now on we abbreviate ``regular balanced Cayley map" to ``RBCM".
The following proposition collects some well-known facts, for which one may refer to \cite{RSJTW05,SS92,WF05}.
\begin{prop}  \label{prop:RBCM}
{\rm(a)} A Cayley map $\mathcal{CM}(\Gamma,\Omega,\rho)$ is a RBCM if and only if $\rho$ extends to an isomorphism of $\Gamma$.

{\rm (b)} For a RBCM $\mathcal{CM}(\Gamma,\Omega=\{\omega_{i}\colon 1\leq i\leq m\},\rho)$ with $\rho(\omega_{i})=\omega_{i+1}$, all of the elements of $\Omega$ have the same order, and
\begin{enumerate}
  \item[\rm(I)] either $m=2n$ and $\omega_{i+n}=\omega_{i}^{-1}$ for all $i$,
  \item[\rm (II)] or all the $\omega_{i}$'s are involutions, (i.e., have order 2).
\end{enumerate}
A RBCM in the case {\rm (I)} or {\rm (II)} is said to be of type {\rm I} or {\rm II}, and denoted by I-RBCM or II-RBCM, respectively.

{\rm (c)} If $\mathcal{CM}(\Gamma,\Omega,\rho)$ and $\mathcal{CM}(\Gamma',\Omega',\rho')$ are two RBCM's of the same type, then they are isomorphic if and only if there exists an isomorphism $f:\Gamma\to \Gamma'$ such that $f(\Omega)=\Omega'$ and $f\circ\rho=\rho'\circ f$.
\end{prop}

A RBCM on a group $\Gamma$ is not only a combinatorial object with good symmetry property, but also can be considered as an extra structure on $\Gamma$.
So far, RBCM's have been classified for cyclic, dihedral, generalized quaternion, dicyclic, and semi-dihedral groups; see \cite{WF05, KO08, Oh09}. Recently, the first author \cite{Ch14} classified RBCM's for several subclasses of abelian $p$-groups.

In this paper, we classify RBCM's on ${\rm PSL}(2,p)$ with $p>3$ a prime number.
This is the first time to obtain concrete results for a family of simple groups.
Two key ingredients are involved in our idea:
(i) each automorphism of ${\rm PSL}(2,p)$ is the conjugation by a unique element of ${\rm PGL}(2,p)$, and
(ii) maximal subgroups of ${\rm PSL}(2,p)$ are known, as recalled in Proposition \ref{prop:max-subgrp}.
By Proposition \ref{prop:RBCM}, to classify $2n$-valent I-RBCM's on ${\rm PSL}(2,p)$, it suffices to find all pairs $(\sigma,\omega)$ with $\sigma\in{\rm PGL}(2,p)$ and $\omega\in{\rm PSL}(2,p)$ such that $\sigma$ has order $2n$, $\sigma^{n}\omega\sigma^{-n}=\omega^{-1}$ and $\sigma^{i}\omega\sigma^{-i}, i=1,\ldots,n$ generate ${\rm PSL}(2,p)$, (these conditions ensure that the conjugacy class of $\omega$ under $\sigma$ has size $2n$). Note that the last condition is equivalent to that the subgroup generated by $\sigma^{i}\omega\sigma^{-i}, i=1,\ldots,n$ is not contained in any maximal subgroup of ${\rm PSL}(2,p)$.
The RBCM's determined by two pairs $(\sigma,\omega)$ and $(\sigma',\omega')$ are isomorphic if and only if there exists $\tau\in{\rm PGL}(2,p)$ such that $\sigma'=\tau\sigma\tau^{-1}$ and $\omega'=\tau\omega\tau^{-1}$. The method for classifying II-RBCM's is similar.

The content is organized as follows. In Section 2 we recall some well-known facts about ${\rm PSL}(2,p)$ and ${\rm PGL}(2,p)$. In Section 3 and 4 we classify I-RBCM's and II-RBCM's on ${\rm PSL}(2,p)$, respectively; in each section we separately deal with the cases $p=5$ and $p>5$, because ${\rm PSL}(2,5)=A_{5}$ plays a special role in subgroup structure of ${\rm PSL}(2,p)$.

\begin{nota}
\rm For a RBCM $\mathcal{CM}({\rm PSL}(2,p),\Omega,\rho)$, if $\Omega=\{\sigma^{i}\omega\sigma^{-i}\colon 1\leq i\leq m\}$ with $\sigma\in{\rm PGL}(2,p)$, then  we denote $\mathcal{CM}({\rm PSL}(2,p),\Omega,\rho)$ by $\mathcal{CM}_{\omega}^{\sigma}$, with the understanding that the permutation $\rho$ is given by $\sigma^{i}\omega\sigma^{-i}\mapsto\sigma^{i+1}\omega\sigma^{-(i+1)}$.

For a set $X$, denote its cardinality by $\#X$.

For an element $\mu$ of a group $\Gamma$, denote its order by $|\mu|$. Given a set of elements $\mu_{1},\ldots,\mu_{\ell}\in \Gamma$, denote the subgroup they generate by $\langle \mu_{1},\ldots,\mu_{\ell}\rangle$.

For two permutations $\omega$ and $\psi$, use $\omega\psi$ to mean ``first do $\omega$, then do $\psi$". For instance, $(12)(23)=(132)$.

Denote the $2\times 2$ identity matrix by $\varepsilon$.
\end{nota}

\section{Preliminary}

For a finite field $\mathbb{F}$, let $\mathbb{F}^{\times}$ denote the multiplicative group of units.
Consider $\mathbb{F}_{p}$ as a subfield of $\mathbb{F}_{p^{2}}$. Fix a generator $e$ of the cyclic group $\mathbb{F}_{p}^{\times}$ and fix a square root $\sqrt{e}\in\mathbb{F}_{p^{2}}^{\times}$, then elements of $\mathbb{F}_{p^{2}}$ are linear combinations $a+b\sqrt{e}$ with $a,b\in\mathbb{F}_{p}$.
The {\it norm}
\begin{align*}
\mathcal{N}&:\mathbb{F}_{p^{2}}^{\times}\to\mathbb{F}_{p}^{\times}, \\
\mathcal{N}(a+b\sqrt{e})&=a^{2}-eb^{2}=(a+b\sqrt{e})^{p+1},
\end{align*}
is a surjective homomorphism (see Problem 1 on Page 87 of \cite{Mo96}).

Fix a generator $w_{1}+w_{2}\sqrt{e}$ of the cyclic group $\mathbb{F}_{p^{2}}^{\times}$ and let $w=w_{1}/w_{2}$, then $w^{2}-e$ has no square root in $\mathbb{F}_{p}$.

Let
\begin{align}
\alpha=\left(\begin{array}{cc} e & 0 \\ 0 & 1 \\ \end{array} \right), \ \ \ \
\beta=\left(\begin{array}{cc} 1 & 1 \\ 0 & 1 \\ \end{array} \right), \ \ \ \
\gamma=\left(\begin{array}{cc} w & e \\ 1 & w \\ \end{array} \right).
\end{align}
It is well-known that (one may refer to \cite{FH50}, Page 68) each element of ${\rm PGL}(2,p)$ is conjugate to $\alpha^{k}$ for some $k$, or $\beta$, or $\gamma^{\ell}$ for some $\ell$.
Furthermore, $\alpha^{k'}$ is conjugate to $\alpha^{k}$ if and only if $k'=\pm k$, and $\gamma^{\ell'}$ is conjugate to $\gamma^{\ell}$ if and only if $\ell'=\pm\ell$.

The following enables us to conveniently deal with the orders of elements of ${\rm PSL}(2,p)$, and can be proved by repeatedly applying
Hamilton-Cayley Theorem $\tilde{\eta}^{2}=t\tilde{\eta}-\varepsilon$.
\begin{prop} \label{prop:order}
Let $\tilde{\eta}\in{\rm SL}(2,p)$ with ${\rm tr}(\tilde{\eta})=t$, and let $\eta\in{\rm PSL}(2,p)$ denote the image of $\tilde{\eta}$ under the quotient homomorphism ${\rm SL}(2,p)\to{\rm PSL}(2,p)$.
\begin{enumerate}
  \item[\rm(a)] $|\eta|=2$ if and only if $t=0$,
  \item[\rm(b)] $|\eta|=3$ if and only if $t^{2}=1$,
  \item[\rm(c)] $|\eta|=4$ if and only if $t^{2}=2$,
  \item[\rm(d)] $|\eta|=5$ if and only if $(t^{2}-1)^{2}=t^{2}$.
\end{enumerate}
\end{prop}
The following result is quoted from Proposition 2.1 of \cite{FJ11}; also see \cite{Di01}.
\begin{prop}  \label{prop:max-subgrp}
Suppose $p\geq 5$. Then each maximal subgroup of ${\rm PSL}(2,p)$ has one of the following forms:
\begin{enumerate}
  \item[\rm (i)] the stabilizer of a point on the projective line $\mathbb{P}^{1}(\mathbb{F}_{p})$;
  \item[\rm (ii)] $D_{p\pm 1}$, the dihedral group of order $p\pm 1$;
  \item[\rm (iii)] $A_{4}$, $S_{4}$ or $A_{5}$.
\end{enumerate}
\end{prop}

\begin{rmk}
\rm A subgroup of form (i) is the same as one whose elements have a common eigenvector.

A subgroup of form (ii) means one isomorphic to $D_{p\pm 1}$, and similarly for (iii). Subgroups in (ii) or (iii) do not always exist, and even when they exist, they may not be maximal.
\end{rmk}

Finally, recall some facts about $S_{4}$, ${\rm PSL}(2,5)=A_{5}$ and ${\rm PGL}(2,5)=S_{5}$:
\begin{prop} \label{prop:A5}
\begin{enumerate}
  \item[\rm(a)] Nontrivial conjugacy classes of $S_{5}$ are listed below $($using $[\mu]$ to denote the conjugacy class containing $\mu)$:
       \begin{align*}
       [(12345)], \ \ \  [(123)], \ \ \  [(12)(34)], \ \ \  [(12)(345)], \ \ \ [(1234)], \ \ \ [(12)],
       \end{align*}
       where the first three classes are contained in $A_{5}$.
  \item[\rm(b)] $S_{4}$ has a presentation $\langle X,Y|X^{2},Y^{3},(XY)^{4}\rangle$, so any group generated by two elements $\mu,\eta$ with $|\mu|=2$, $|\eta|=3$ and $|\mu\eta|=4$ is a quotient of $S_{4}$.
  \item[\rm(c)] $A_{5}$ has a presentation $\langle X,Y|X^{2},Y^{3},(XY)^{5}\rangle$, so any nontrivial group generated by two elements $\mu,\eta$ with $|\mu|=2$, $|\eta|=3$ and $|\mu\eta|=5$ is isomorphic to $A_{5}$.
  \item[\rm(d)] Each non-abelian proper subgroup of $A_{5}$ is isomorphic to $D_{6}$, $D_{10}$ or $A_{4}$.
\end{enumerate}
\end{prop}


We explain (d). Let $\Gamma\leq A_{5}$ be a non-abelian proper subgroup, then $\#\Gamma\in\{6,10,12,15,20\}$. Clearly $\Gamma$ is dihedral if $\#\Gamma\in\{6,10\}$.
By  \cite{Ar04} Theorem 7.8.1 which classifies groups of order 12, $\Gamma=A_{4}$ if $\#\Gamma=12$. By \cite{Ar04} Theorem 7.7.7 (a), each group of order 15 is cyclic, so $\#\Gamma$ never equals $15$. Finally, if $\#\Gamma=20$, then by Sylow's Theorem, $\Gamma$ has exactly one subgroup of order 5, so all the other 15 elements are involutions, which are exactly all the 15 involutions in $A_{5}$, but the product of $(12)(34)$ and $(12)(35)$ is $(345)$, whose order cannot divide 20.

\section{Type I regular balanced Cayley maps}

\subsection{I-RBCM's on ${\rm PSL}(2,5)=A_{5}$} \label{sec:A5-1}

Suppose $\mathcal{CM}_{\omega}^{\sigma}$ is a I-RBCM on ${\rm PSL}(2,5)=A_{5}$, with $\omega\in A_{5}$, $|\omega|>2$, $\sigma\in S_{5}$, $|\sigma|=2n$. Clearly $2<2n\leq 6$, hence $n=2$ or $n=3$. We may assume that $\omega$ is one of representatives of conjugacy classes as in Proposition \ref{prop:A5} (a), namely, $\omega=(123)$ or $\omega=(12345)$.
Denote $\omega_{i}=\sigma^{i}\omega\sigma^{-i}$, $1\leq i\leq 2n$.

There are four possibilities.
\begin{itemize}
  \item If $n=2$ and $\omega=(123)$, then it follows from $\omega_{2}=\omega^{-1}$ that $\sigma=(k4\ell 5)$
        with $\{k,\ell\}\subset\{1,2,3\}$. We may find $\tau\in\{(1),\omega,\omega^{2}\}$ such that $\tau\sigma\tau^{-1}=(1425)$
        and $\mathcal{CM}^{\tau\sigma\tau^{-1}}_{\omega}\cong\mathcal{CM}^{\sigma}_{\omega}$. Just assume $\sigma=(1425)$.
        Now $\omega_{1}=(543)$, $\omega\omega_{1}\omega\omega_{1}^{-1}=(14)(23)$ has order 2, and $(\omega\omega_{1}\omega\omega_{1}^{-1})\omega_{1}=\omega\omega_{1}\omega=(13254)$ has order 5, thus
        $\langle\omega\omega_{1}\omega\omega_{1}^{-1}, \omega_{1}\rangle=A_{5}$ and also $\langle\omega_{1},\omega_{2}\rangle=A_{5}$.

  \item If $n=2$ and $\omega=(12345)$, then by replacing $\sigma$ by $\omega^{k}\sigma\omega^{-k}$ for some $k$
        if necessary, we may assume $\sigma$ fixes the letter $5$. Hence $\sigma=(1243)$ or $\sigma=(1342)$, due to the condition $\omega_{2}=\omega^{-1}$.
        If $\sigma=(1243)$, then $\omega_{1}=(31425)=\omega^{-2}$; if $\sigma=(1342)$, then $\omega_{1}=(24135)=\omega^{2}$.
        In neither case $\langle\omega,\sigma\omega\sigma^{-1}\rangle=A_{5}$, as $A_{5}$ is not cyclic.
  \item If $n=3$ and $\omega=(123)$, then $\sigma=(k_{1}k_{2})(\ell_{1}\ell_{2}\ell_{3})$ with $\{k_{1},k_{2}\}\subset\{1,2,3\}$ due to $\omega_{3}=\omega^{-1}$. By replacing $\sigma$ by $\omega^{k}\sigma\omega^{-k}$ for some $k$
        if necessary, we may assume $\sigma=(12)(345)$, hence $\omega_{1}=(215)$, $\omega_{2}=(124)$. Now $\omega\omega_{1}^{-1}=(15)(23)$, $(\omega\omega_{1}^{-1})\omega_{2}=(15234)$, so $\langle \omega\omega_{1}^{-1}, \omega_{2}\rangle=A_{5}$ and also $\langle\omega_{1},\omega_{2},\omega_{3}\rangle=A_{5}$.

  \item If $n=3$ and $\omega=(12345)$, then it is impossible that $\sigma^{3}\omega\sigma^{-3}=\omega^{-1}$, since $\sigma^{3}$ is a transposition.
\end{itemize}

\begin{thm} \label{thm:I-RBCM-A5}
Each $4$-valent I-RBCM on $A_{5}$ is isomorphic to $\mathcal{CM}_{(123)}^{(1425)}$,
each $6$-valent type I RBCM on $A_{5}$ is isomorphic to $\mathcal{CM}_{(123)}^{(12)(345)}$,
and there does not exist a $2n$-valent I-RBCM for $n>3$.
\end{thm}

\subsection{I-RBCM's on ${\rm PSL}(2,p)$ for $p>5$}

Suppose $\mathcal{CM}_{\omega}^{\sigma}$ is a $2n$-valent I-RBCM with $|\sigma|=2n\geq 4$, noting that ${\rm PSL}(2,p)$ is not cyclic.
Now that $\sigma$ is conjugate to $\alpha^{k}$ or $\gamma^{\ell}$, we may just assume $\sigma=\alpha^{k}$ with $1\leq k\leq (p-1)/2$, or $\sigma=\gamma^{\ell}$ with $1\leq\ell\leq(p+1)/2$.

Setting
\begin{align}
\tau=\left\{\begin{array}{cc}
\varepsilon, \ \ \ &\text{ if\ } \sigma=\alpha^{k}, \\
\left(\begin{array}{cc} \sqrt{e} & -e \\ \sqrt{e} & e \\ \end{array}\right), \ \ \ &\text{ if\ } \sigma=\gamma^{\ell}, \end{array}\right. \label{eq:tau}
\end{align}
one has
\begin{align}
\tau\sigma\tau^{-1}=\left(\begin{array}{cc} s & 0 \\ 0 & 1/s \\ \end{array} \right)\in{\rm PGL}(2,p^{2}), \ \ \ \text{with\ }
s=\left\{\begin{array}{cc}
(\sqrt{e})^{k}, &\text{if\ } \sigma=\alpha^{k}, \\
\left(\frac{w-\sqrt{e}}{\sqrt{w^{2}-e}}\right)^{-\ell}, &\text{if\ } \sigma=\gamma^{\ell}; \\
\end{array}\right.   \label{eq:s}
\end{align}
note that $s$ has order $4n$ as an element of $\mathbb{F}_{p^{2}}^{\times}$, hence $s^{2n}=-1$.

Suppose
\begin{align}
\tau\omega\tau^{-1}=\left(\begin{array}{cc} a & b \\ c & d \\ \end{array} \right)\in{\rm PSL}(2,p^{2}),  \ \ \ \ \text{with\ \ } a^{2}-bc=1. \label{eq:psi1}
\end{align}
The condition $\sigma^{n}\omega\sigma^{-n}=\omega^{-1}$ is equivalent to
$\left(\begin{array}{cc} a & -b \\ -c & d \\ \end{array}\right)
=\left(\begin{array}{cc} d & -b \\ -c & a \\ \end{array} \right)$, implying $a=d$.

\begin{lem}  \label{lem:I}
The elements $\sigma^{i}\omega\sigma^{-i}, i=1,\ldots,n$ generate ${\rm PSL}(2,p)$ if and only if
$abc\neq 0$ and $2a^{2}\neq 1$ when $n=2$.
\end{lem}

\begin{proof}
Let $\Gamma=\langle\psi_{1},\ldots,\psi_{n}\rangle$ with $\psi_{i}=\tau\sigma^{i}\omega\sigma^{-i}\tau^{-1}$.
The task is to show $\tau^{-1}\Gamma\tau={\rm PSL}(2,p)$.
We have
\begin{align}
\psi_{i}=\left(\begin{array}{ll} a & s^{2i}b \\ s^{-2i}c & a \\ \end{array}\right), \ \ \ \ \ \
\psi_{i}\psi_{j}=\left(\begin{array}{ll}a^{2}+s^{2(i-j)}bc & (s^{2i}+s^{2j})ab \\ (s^{-2i}+s^{-2j})ac & a^{2}+s^{2(j-i)}bc \\ \end{array}\right). \label{eq:product}
\end{align}

If $a=0$, then $\psi_{i}$ is counter-diagonal,
so each element of $\Gamma$ is either diagonal or counter-diagonal, hence
$$\#\Gamma\leq 2(p^{2}-1)<p(p^{2}-1)/2=\#{\rm PSL}(2,p).$$
If $bc=0$, then $\psi_{i}$ is unipotent, hence also $\#\Gamma<\#{\rm PSL}(2,p)$.

Thus a necessary condition for $\tau^{-1}\Gamma\tau={\rm PSL}(2,p)$ is $abc\neq 0$. We show that this is also sufficient except for the case
when $n=2$ and $2a^{2}=1$.

Suppose $abc\neq 0$. Then $|\psi_{i}|>2$, as ${\rm tr}(\psi_{i})=2a\neq 0$.

(a) If $\Gamma\leq D_{2m}$, then $\psi_{i}\in \mathbb{Z}/m\mathbb{Z}$, but by Eq.(\ref{eq:product}) $\Gamma$ is not abelian.

(b) If $\Gamma$ is contained in a subgroup of form (i) in Proposition \ref{prop:max-subgrp}, then $\psi_{1}$ and $\psi_{2}$ have a common eigenvector $(x,y)^{t}\in\mathbb{F}_{p^{2}}^{2}$, hence both $(x,y)$ and $(sx,y)^{t}$ are eigenvectors of $\psi_{1}$; this implies $x=0$ or $y=0$, which contradicts the assumption that $bc\neq 0$.

(c) If $\Gamma\leq S_{4}$, then $\Gamma=A_{4}$ or $S_{4}$ according to (a).
It is well-known that each automorphism of $A_{4}$ or $S_{4}$ is the conjugation by some element in $S_{4}$, whose order belongs to $\{2,3,4\}$, hence $n=2$, $s^{4}=-1$ and ${\rm tr}(\psi_{1}\psi_{2})=2a^{2}$, using Eq.(\ref{eq:product}).
If $|\psi_{i}|=3$, then $\Gamma=A_{4}$, and $4a^{2}=({\rm tr}(\psi_{1}))^{2}=1$, so ${\rm tr}(\psi_{1}\psi_{2})=1/2$, implying $|\psi_{1}\psi_{2}|\neq 2,3$, but this contradicts $\psi_{1}\psi_{2}\in A_{4}$. Thus $|\psi_{i}|=4$, and $2a^{2}-1={\rm tr}(\psi_{2}^{2})=0$, i.e., $2a^{2}=1$.

Conversely, if $n=2$ and $2a^{2}=1$, then (denoting $\tau\sigma\tau^{-1}$ by $\varsigma$)
$$({\rm tr}(\varsigma\psi_{4}))^{2}=(s+s^{-1})^{2}a^{2}=1, \ \ \ \ \ \
{\rm tr}(\varsigma^{2}\psi_{4})=(s^{2}+s^{-2})a^{2}=0,$$
hence $|\varsigma\psi_{4}|=3$ and $|\psi_{4}^{-1}\varsigma^{-2}|=|\varsigma^{2}\psi_{4}|=2$; this together with $|(\psi_{4}^{-1}\varsigma^{-2})(\varsigma\psi_{4})|=|\varsigma|=4$ implies that
$\langle\varsigma,\psi_{4}\rangle=\langle\psi_{4}^{-1}\varsigma^{-2},\varsigma\varpi\rangle$ is a quotient of $S_{4}$. Thus $\#\Gamma\leq 24$.

(d) If $\Gamma\leq A_{5}$, then $\Gamma=A_{5}$ by (a) and Proposition \ref{prop:A5} (d).
By Theorem \ref{thm:I-RBCM-A5}, there are two possibilities; in both cases $|\psi_{i}|=3$ hence $4a^{2}=1$.
  \begin{enumerate}
    \item[\rm(i)] $n=2$ (so that $s^{4}=-1$), then ${\rm tr}(\psi_{1}\psi_{2})=2a^{2}\neq 0,\pm 1$ and is not a root of $(t^{2}-1)^{2}=t^{2}$, hence $|\psi_{1}\psi_{2}|\notin\{2,3,5\}$, contradicting $\psi_{1}\psi_{2}\in A_{5}$.

   \item[\rm(ii)] $n=3$ (so that $s^{4}-s^{2}+1=0$), then there exists an isomorphism $A_{5}\cong\Gamma$ sending $(215)$ to $\psi_{1}$ and $(124)$ to $\psi_{2}$, hence it sends $(154)=(215)(124)$ to $\psi_{1}\psi_{2}$. But ${\rm tr}(\psi_{1}\psi_{2})=3a^{2}-1=-1/4\neq\pm 1$, a contradiction.
  \end{enumerate}
\end{proof}

If $\sigma=\alpha^{k}$, then $2n\mid p-1$ and $k=(p-1)u/2n$ for some $u$ with $(u,2n)=1$, $1\leq u<n$.
Now
$\varphi:=\left(\begin{array}{cc} b^{-1} & 0 \\ 0 & 1 \\ \end{array}\right)$ commutes with $\alpha^{k}$ and
\begin{align}
\varphi\omega\varphi^{-1}=\left(\begin{array}{cc} a & 1 \\ a^{2}-1 & a \\ \end{array}\right)=:\omega({\rm i},a),
\end{align}
thus $\mathcal{CM}_{\omega}^{\alpha^{k}}\cong\mathcal{CM}_{\omega({\rm i},a)}^{\alpha^{k}}$.
Furthermore, for $a\neq a'$, $\omega({\rm i},a)$ is not conjugate to $\omega({\rm i},a')$,  as ${\rm tr}(\omega({\rm i},a))\neq {\rm tr}(\omega({\rm i},a'))$,
so $\mathcal{CM}_{\omega({\rm i},a)}^{\alpha^{k}}\ncong\mathcal{CM}_{\omega({\rm i},a')}^{\alpha^{k}}$.

\medskip

If $\sigma=\gamma^{\ell}$, then $2n\mid p+1$ and $\ell=(p+1)v/2n$ for some $v$ with $(v,2n)=1$, $1\leq v<n$. Note that, if $n=2$, then $2a^{2}$ never equals $1$ since the Legendre symbol $(2/p)=-1$ by Theorem 1 (b) on Page 53 of \cite{IR90}.
By Eq.(\ref{eq:tau}) and (\ref{eq:psi1}),
\begin{align}
\omega=\left(\begin{array}{cc} a+(b+c)/2 & \sqrt{e}(b-c)/2 \\ (c-b)/2\sqrt{e} & a-(b+c)/2 \\ \end{array} \right)=
\left(\begin{array}{cc} x & -ez \\ z & 2a-x \\ \end{array} \right)\in{\rm PSL}(2,p),
\end{align}
where $x=a+(b+c)/2$ and $z=(c-b)/2\sqrt{e}$ are elements of $\mathbb{F}_{p}$, hence
\begin{align}
(x-a)^{2}-ez^{2}=bc=a^{2}-1,  \label{eq:x-z}
\end{align}
i.e., $\mathcal{N}(x-a+z\sqrt{e})=a^{2}-1$. When $a$ is fixed, this equation has $p+1$ solutions, since the homomorphism $\mathcal{N}:\mathbb{F}_{p^{2}}^{\times}\to\mathbb{F}_{p}^{\times}$ is surjective.
Choose and fix a solution $(x_{{\rm i},a},z_{{\rm i},a})$, and put
\begin{align}
\tilde{\omega}({\rm i},a)=\left(\begin{array}{cc} x_{{\rm i},a} & -ez_{{\rm i},a} \\ z_{{\rm i},a} & 2a-x_{{\rm i},a} \\ \end{array}\right).
\end{align}
Noticing
$$\tau\gamma\tau^{-1}=\left(\begin{array}{cc} w-\sqrt{e} & 0 \\ 0 & w+\sqrt{e} \\ \end{array}\right), \ \ \
\text{and\ }\tau\tilde{\omega}({\rm i},a)\tau^{-1}=\left(\begin{array}{cc} a & b' \\ c' & -a \\ \end{array}\right)$$
for some $b',c'$ with $b'c'=a^{2}-1$, we easily see that the $p+1$ elements
$$\gamma^{h}\tilde{\omega}(i,a)\gamma^{-h}, \ \ \ \ h=1,\ldots,p+1$$
are distinct from each other.
Thus for the present $\omega$, there exist (a unique) $h\in\{1,\ldots,p+1\}$ such that $\gamma^{h}\tilde{\omega}(i,a)\gamma^{-h}=\omega$, and hence
$\mathcal{CM}^{\gamma^{\ell}}_{\omega}\cong\mathcal{CM}^{\gamma^{\ell}}_{\tilde{\omega}(i,a)}$.

\begin{thm}
Suppose $\mathcal{M}$ is a $2n$-valent I-RBCM on ${\rm PSL}(2,p)$ with $p>5$, then $2n\mid p-1$ or $2n\mid p+1$.
\begin{enumerate}
  \item[\rm(i)] If $2n\mid p-1$, then $\mathcal{M}\cong\mathcal{CM}_{\omega({\rm i},a)}^{\alpha^{(p-1)u/2n}}$ for a unique pair $(a,u)$ such that
                $a\notin\{\pm 1, 0\}$, $(u,2n)=1$, $1\leq u<n$, and moreover, $2a^{2}\neq 1$ if $n=2$.
  \item[\rm(ii)] If $2n\mid p+1$, then $\mathcal{M}\cong\mathcal{CM}_{\tilde{\omega}({\rm i},a)}^{\gamma^{(p+1)v/2n}}$ for a unique
                 pair $(a,v)$ such that $a\notin\{\pm1, 0\}$, $(v,2n)=1$ and $1\leq v<n$.
\end{enumerate}
\end{thm}

\section{Type II regular balanced Cayley maps}

\subsection{II-RBCM's on ${\rm PSL}(2,5)=A_{5}$} \label{sec:A5-2}

Suppose $\mathcal{CM}_{\omega}^{\sigma}$ is an $n$-valent II-RBCM on ${\rm PSL}(2,5)=A_{5}$, with $\omega\in A_{5}$, $|\omega|=2$, and $\sigma\in S_{5}$, $|\sigma|=n\leq 6$. Since a nonabelian group generated by two involutions must be dihedral, we have $n>2$. Also note that the action of the conjugation by $\sigma$ on the set of involutions of $A_{5}$ has at most one fixed element, hence $n\mid 15$ or $n\mid 14$ which implies $n=3$ or $n=5$.

Denote $\omega_{i}=\sigma^{i}\omega\sigma^{-i}, i=1,\ldots,n$, and denote $\Gamma=\langle\omega_{1},\ldots,\omega_{n}\rangle$.

If $n=3$, we may assume $\sigma=(123)$ and $\omega$ fixes the letter $1$. The condition $\langle\omega_{1},\omega_{2},\omega_{3}\rangle=A_{5}$ requires $\omega\in\{(24)(35),(25)(34)\}$. The conjugation by $(45)$ fixes $\sigma$ and takes $(25)(34)$ to $(24)(35)$, so let us just assume $\omega=(24)(35)$. Then $\omega_{1}=(14)(25)$, $\omega_{2}=(15)(34)$, $\omega_{1}\omega_{2}\omega_{1}\omega_{3}=(153)$, $\omega_{1}(\omega_{1}\omega_{2}\omega_{1}\omega_{3})=\omega_{2}\omega_{1}\omega_{3}=(14523)$, thus $\langle\omega_{1},\omega_{1}\omega_{2}\omega_{1}\omega_{3}\rangle =A_{5}$, and also $\Gamma=A_{5}$.

If $n=5$, we may assume $\sigma=(12345)$. There are three possibilities.
\begin{enumerate}
  \item[\rm(i)] If $\omega_{5}=(12)(34)$, then $\omega_{1}=(15)(23)$, $\omega_{2}=(12)(45)$, $\omega_{3}=(15)(34)$, $\omega_{4}=(23)(45)$, so
                $\omega_{3}\omega_{5}=(152)$, $\omega_{5}\omega_{4}=(13542)$, $\omega_{5}\omega_{4}\omega_{5}\omega_{3}=(13)(45)$, hence
                $\langle\omega_{5}\omega_{4}\omega_{5}\omega_{3},\omega_{3}\omega_{5}\rangle= A_{5}$ and also $\Gamma=A_{5}$.
  \item[\rm(ii)] If $\omega_{5}=(13)(24)$, then $\omega_{1}=(13)(25)$, $\omega_{2}=(14)(25)$, $\omega_{3}=(14)(35)$, $\omega_{4}=(24)(35)$, so
                 $\omega_{4}\omega_{5}=(135)$, $\omega_{5}\omega_{2}=(13452)$, $\omega_{5}\omega_{2}\omega_{5}\omega_{4}=(25)(34)$, hence
                 $\langle\omega_{5}\omega_{2}\omega_{5}\omega_{4},\omega_{4}\omega_{5}\rangle= A_{5}$ and also $\Gamma=A_{5}$.
  \item[\rm(iii)] If $\omega_{5}=(14)(23)$, then $\omega_{1}=(12)(35)$, $\omega_{2}=(15)(24)$, $\omega_{3}=(13)(45)$, $\omega_{4}=(25)(34)$,
                  so $\omega_{4}\omega_{5}=(14253)$. One can check that $\omega_{i}=\omega_{5}(\omega_{4}\omega_{5})^{i}$, $i=1,\ldots,5$, hence  $\Gamma=\langle\omega_{4},\omega_{5}\rangle\cong D_{10}$, which is impossible.
\end{enumerate}

\begin{thm}  \label{thm:II-RBCM-A5}
Each $3$-valent II-RBCM on $A_{5}$ is isomorphic to
$\mathcal{CM}_{(24)(35)}^{(123)}$,
each $5$-valent II-RBCM on $A_{5}$ is isomorphic to
$\mathcal{CM}_{(12)(34)}^{(12345)}$ or $\mathcal{CM}_{(13)(24)}^{(12345)}$,
and there does not exist an $n$-valent II-RBCM on $A_{5}$ for $n\neq 3,5$.
\end{thm}

\subsection{II-RBCM's on ${\rm PSL}(2,p)$ for $p>5$}

Let $\mathcal{CM}_{\omega}^{\sigma}$ be an $n$-valent II-RBCM on ${\rm PSL}(2,p)$, with $|\omega|=2$ and $|\sigma|=n>2$
(as ${\rm PSL}(2,p)$ is not dihedral).
We may assume $\sigma$ is equal to $\alpha^{k}$, $\beta$ or $\gamma^{\ell}$.
Suppose
\begin{align}
\omega=\left( \begin{array}{cc} x & y \\ z & -x \\ \end{array} \right)  \ \ \ \ \text{with\ \ } x^{2}+yz=-1.  \label{eq:omega}
\end{align}

If $n=p$, then $\sigma=\beta$, and
\begin{align}
\beta^{i}\omega\beta^{-i}=
\left(\begin{array}{cc} x+iz & y-2ix-i^{2}z \\ z & -x-iz \\ \end{array} \right), \ \ \ \ 1\leq i\leq p.
\end{align}

\begin{lem}  \label{lem:II1}
The elements $\beta^{i}\omega\beta^{-i}, i=1,\ldots,p$ generate ${\rm PSL}(2,p)$ if and only if $z\neq 0$.
\end{lem}

\begin{proof}
Let $\Gamma=\langle\omega_{1},\ldots,\omega_{p}\rangle$ with $\omega_{i}=\beta^{i}\omega\beta^{-i}$.
If $z=0$, then each $\omega_{i}$ is upper-triangular, and so is each element of $\Gamma$, hence $\Gamma\neq{\rm PSL}(2,p)$.

Suppose $z\neq 0$, we shall prove that $\Gamma={\rm PSL}(2,p)$.

Firstly, $\Gamma$ is not contained in the stabilizer of any point of $\mathbb{P}^{1}(\mathbb{F}_{p})$: if $\xi\in\mathbb{F}_{p}^{2}$ is a common eigenvector of $\omega_{1}$ and $\omega_{2}$, then $\xi$ and $\beta\xi$ are both eigenvectors of $\omega_{1}$, hence $\xi$ and $\beta\xi$ are linearly dependent, which is impossible when $z\neq 0$.

Secondly, we show that $\Gamma$ is not contained in any other maximal subgroup, by counting involutions.
For a group $\Delta$, let $I(\Delta)$ denote the set of involutions of $\Delta$.
It is obvious that $\beta^{i}\phi\beta^{-i}\neq\phi$ for any $\phi\in{\rm PSL}(2,p)$ unless $\phi$ is upper-unitriangular, in which case $\phi\notin I(\Gamma)$, so $\langle\beta\rangle$ acts freely on $I(\Gamma)$ by conjugation. Thus
$$(\star) \ \ \ \ \ \ p\mid\#I(\Gamma) \ \ \ \text{and\ \ \ } \#I(\Gamma)\geq p. $$
\begin{itemize}
  \item It is impossible that $\Gamma\leq D_{p-1}$, since $\#I(D_{p-1})=p-1<p$.

        If $\Gamma\leq D_{p+1}$, then $\Gamma=D_{p+1}$ since any $p$ involutions generate $D_{p+1}$; but $\#I(D_{p+1})=p+1$
        is not a multiple of $p$, contradicting ($\star$).
  \item It is impossible that $\Gamma\leq A_{4}$, since $\#I(A_{4})=3<p$.
  \item If $\Gamma\leq S_{4}$, then $\#I(\Gamma)=p=7$ since $\#I(S_{4})=9$, but no subgroup of $S_{4}$ contains exactly 7 involutions.
  \item If $\Gamma\leq A_{5}$, then $\#I(\Gamma)=p\in\{7,11,13\}$ since $\#I(A_{5})=15$, but $A_{5}$ does not have such a subgroup $\Gamma$,
        as can be checked.
\end{itemize}
\end{proof}

Let $m$ be an integer whose residue class modulo $p$ is $-x/z$, then
\begin{align}
\beta^{m}\omega\beta^{-m}=\left(\begin{array}{cc} 0 & -1/z  \\ z & 0 \\ \end{array}\right)=:\varpi(z),
\end{align}
hence $\mathcal{CM}_{\omega}^{\beta}\cong\mathcal{CM}_{\varpi(z)}^{\beta}$. Furthermore, for $z\neq z'$, there does not exist $\tau$ with $\tau\beta\tau^{-1}=\beta$ and $\tau\varpi(z)\tau^{-1}=\varpi(z')$, so $\mathcal{CM}_{\varpi(z)}^{\beta}\ncong\mathcal{CM}_{\varpi(z')}^{\beta}$.

\bigskip

If $n\neq p$, then take $\tau$ as in Eq.(\ref{eq:tau}), so that
$$\varsigma:=\tau\sigma\tau^{-1}=\left(\begin{array}{ll} s & 0 \\ 0 & 1/s \\ \end{array}\right)\in{\rm PGL}(2,p^{2})$$
with $s$ given by Eq.(\ref{eq:s}).
Suppose
\begin{align}
\tau\omega\tau^{-1}=\left(\begin{array}{cc} a & b \\ c & -a \\ \end{array} \right)\in{\rm PSL}(2,p^{2}).  \label{eq:psi2}
\end{align}

\begin{lem}  \label{lem:II2}
The elements $\sigma^{i}\omega\sigma^{-i}, i=1,\ldots,n$ generate ${\rm PSL}(2,p)$ if and only if
$abc\neq 0$ and {\rm(i)} $3a^{2}\neq -1,-2$ and $9a^{4}+9a^{2}+1\neq 0$ if $n=3$, {\rm(ii)} $5a^{4}+5a^{2}+1\neq 0$ if $n=5$.
\end{lem}

\begin{proof}
Let $\Gamma=\langle\psi_{1},\ldots,\psi_{n}\rangle$ with $\psi_{i}=\varsigma^{i}(\tau\omega\tau^{-1})\varsigma^{-i}$.
We have
\begin{align}
\psi_{i}&=\left(\begin{array}{cc} a & s^{2i}b \\ s^{-2i}c & -a \\ \end{array} \right), \ \ \ \
\psi_{i}\psi_{j}=\left(\begin{array}{cc} a^{2}+s^{2(i-j)}bc & (s^{2j}-s^{2i})ab \\ (s^{-2i}-s^{-2j})ac & a^{2}+s^{2(j-i)}bc \\ \end{array} \right).
\end{align}

It can be verified that $\psi_{i}, 1\leq i\leq n$ have a common eigenvector if and only if $abc=0$, so a necessary condition for $\tau^{-1}\Gamma\tau={\rm PSL}(2,p)$ is $abc\neq 0$ which we assume below.

(a) If $\Gamma\leq D_{2m}$ for some $m$, then $\psi_{n}\psi_{1}\psi_{2}$ is an involution, hence
        $$0={\rm tr}(\psi_{n}\psi_{1}\psi_{2})=(2(s^{-2}-s^{2})+s^{4}-s^{-4})abc.$$
    But this is impossible since $s^{2}\neq \pm 1$. In particular, $\Gamma\not\leq D_{p\pm 1}$.

(b) If $\Gamma\leq A_{4}$, then since $A_{4}$ has exactly 3 involutions which commute with each other, we have $n=3$ and $\psi_{3}\psi_{1}=\psi_{2}$, hence $s^{4}+s^{2}+1=0$ and
$$\left(\begin{array}{cc} a^{2}+s^{4}bc & (s^{2}-1)ab \\ (1-s^{4})ac & a^{2}+s^{2}bc \\ \end{array}\right)=
\left(\begin{array}{cc} a & s^{4}b \\ s^{2}c & -a \\ \end{array}\right)\in{\rm PSL}(2,p^{2}).$$
This is equivalent to $3a^{2}=-1$, as can be verified.

(c) If $\Gamma\leq S_{4}$ but $\Gamma\nleq A_{4}$, then $\Gamma=S_{4}$ since any other subgroup of $S_{4}$ is dihedral, which cannot contain $\Gamma$ by (a). Recall that each automorphism of $S_{4}$ is the conjugation by some $\eta\in S_{4}$. Suppose the element of $\Gamma$ corresponding to $\eta$ is
$\vartheta$, then for each $i$,
$\vartheta\psi_{i}\vartheta^{-1}=\psi_{i+1}=\varsigma\psi_{i}\varsigma^{-1}$, hence $\varsigma^{-1}\vartheta$ commutes with $\psi_{i}$; this implies that $\vartheta=\varsigma$ since the $\psi_{i}$'s have no common eigenvector. Thus $\varsigma\in\Gamma$.

Below, we write elements of $\Gamma$ and also $\varsigma$ as matrices and permutations simultaneously.
Clearly each $\psi_{i}$ is a transposition; just assume $\psi_{n}=(12)$. There are two possibilities.
\begin{enumerate}
  \item[\rm(i)] If $n=3$, then $\varsigma$ does not fix the letter $4$, hence we may assume $\varsigma=(432)$,
           and then $\psi_{3}\varsigma=(1432)$ has order 4, hence
           $$2=({\rm tr}(\psi_{6}\varsigma))^{2}=(s-s^{-1})^{2}a^{2}=-3a^{2}.$$

           On the other hand, when $n=3$ and $3a^{2}=-2$, one has $|\psi_{3}\varsigma|=4$, so $\langle\psi_{1},\varsigma\rangle$ is a quotient of $S_{4}$. Thus $\#\Gamma\leq 24$.

  \item[\rm(ii)] If $n=4$, then we may assume $\varsigma=(4321)$, so $\psi_{2}=(34)$ commutes with $\psi_{4}$. This implies $(s^{4}-1)ab=0$ which is absurd.
\end{enumerate}

(d) If $\Gamma\leq A_{5}$ and $\Gamma\nleq A_{4}$, then $\Gamma=A_{5}$ by (a) and Proposition \ref{prop:A5} (d).
Arguing similarly as in (c), we can show $\varsigma\in\Gamma$.
By Theorem \ref{thm:II-RBCM-A5}, (up to isomorphism) there are three possibilities.
\begin{enumerate}
  \item[\rm(i)] $n=3$ (so that $s^{4}+s^{2}+1=0$), $\varsigma=(123)$, $\psi_{3}=(24)(35)$. Then $\psi_{3}\varsigma=(12435)$,
        hence
        $$((s-s^{-1})^{2}a^{2}-1)^{2}=(s-s^{-1})^{2}a^{2},$$
        implying $9a^{4}+9a^{2}+1=0$.

        Conversely, if $n=3$ and $9a^{4}+9a^{2}+1=0$, then $|\psi_{3}\varsigma|=5$, hence $\langle\psi_{3},\varsigma\rangle=A_{5}$, and
        $\Gamma\leq A_{5}$.

  \item[\rm(ii)] $n=5$, (so that $s^{4}+s^{2}+s^{-2}+s^{-4}=-1$), $\varsigma=(12345)$, $\psi_{5}=(12)(34)$. Then $\psi_{5}\varsigma=(135)$, hence
        \begin{align}
        1=({\rm tr}(\psi_{5}\varsigma))^{2}=(s-s^{-1})^{2}a^{2},  \label{eq:a1}
        \end{align}

        Conversely, if $n=5$ and Eq.(\ref{eq:a1}) holds, then $|\psi_{5}\varsigma|=3$, $\langle\psi_{5},\varsigma\rangle=\langle\psi_{5},\psi_{5}\varsigma\rangle=A_{5}$, and $\Gamma\leq A_{5}$.

  \item[\rm(iii)] $n=5$, $\varsigma=(12345)$, $\psi_{5}=(13)(24)$. Then $\psi_{5}\varsigma^{2}=(152)$, hence
        \begin{align}
        1=({\rm tr}(\psi_{5}\varsigma^{2}))^{2}=(s^{2}-s^{-2})^{2}a^{2}, \label{eq:a2}
        \end{align}

        Conversely, when $n=5$ and Eq.(\ref{eq:a2}) holds, then $|\psi_{5}\varsigma|=3$, $\langle\psi_{5},\varsigma\rangle=\langle\psi_{5},\psi_{5}\varsigma^{2}\rangle=A_{5}$, and $\Gamma\leq A_{5}$.
\end{enumerate}
Note that $(s-s^{-1})^{2}+(s^{2}-s^{-2})^{2}=-5$ and $(s-s^{-1})^{2}\cdot(s^{2}-s^{-2})^{2}=5$, hence Eq.(\ref{eq:a1}) or Eq.(\ref{eq:a2}) holds if and only if
$1/a^{2}$ is a root of $t^{2}+5t+5=0$.
\end{proof}

If $\sigma=\alpha^{k}$, then $n\mid p-1$ and $k=(p-1)u/n$ for some $u$ with $(u,n)=1$, $1\leq u<n/2$.
Now
$\varphi:=\left(\begin{array}{cc} b^{-1} & 0 \\ 0 & 1 \\ \end{array}\right)$ commutes with $\alpha^{k}$ and
\begin{align}
\varphi\omega\varphi^{-1}=\left(\begin{array}{cc} a & 1 \\ -a^{2}-1 & -a \\ \end{array}\right)=:\omega({\rm ii},a),
\end{align}
hence $\mathcal{CM}_{\omega}^{\sigma}\cong\mathcal{CM}_{\omega({\rm ii},a)}^{\alpha^{k}}$.
Furthermore, for $a\neq a'$, there does not exist $\vartheta$ with $\vartheta\alpha^{k}\vartheta^{-1}=\alpha^{k}$ and
$\vartheta\omega({\rm ii},a)\vartheta^{-1}=\omega({\rm ii},a')$, so
$\mathcal{CM}_{\omega({\rm ii},a)}^{\alpha^{k}}\ncong\mathcal{CM}_{\omega({\rm ii},a')}^{\alpha^{k}}$.

\medskip

If $\sigma=\gamma^{\ell}$, then $n\mid p+1$ and $\ell=(p+1)v/n$ for some $v$ with $(v,n)=1$, $1\leq v<n/2$.
One deduces from Eq.(\ref{eq:tau}), (\ref{eq:omega}) and (\ref{eq:psi2}) that
\begin{align}
y+ez=-2\check{a}, \hspace{20mm} \text{with\ \ } \check{a}=a{\sqrt{e}}\in\mathbb{F}_{p}, \end{align}
and then from $x^{2}+yz=-1$ that
\begin{align}
(y+\check{a})^{2}-ex^{2}=\check{a}^{2}+e.   \label{eq:x-z2}
\end{align}
The condition $bc\neq 0$ is equivalent to $\check{a}^{2}+e\neq 0$. When $\check{a}$ is fixed, the equation (\ref{eq:x-z2}) has $p+1$ solutions

Choose a solution $(x_{{\rm ii},\check{a}},y_{{\rm ii},\check{a}})$, and let
\begin{align}
\tilde{\omega}({\rm ii},\check{a})=\left(\begin{array}{cc} x_{{\rm ii},\check{a}} & y_{{\rm ii},\check{a}} \\ -(2\check{a}+y_{{\rm ii},\check{a}})/e & -x_{{\rm ii},\check{a}} \\ \end{array}\right).
\end{align}
Arguing similarly as in Section 3, we can show that
$\mathcal{CM}_{\omega}^{\gamma^{\ell}}\cong\mathcal{CM}_{\tilde{\omega}({\rm ii},\check{a})}^{\gamma^{\ell}}$.

\begin{thm}
Suppose $\mathcal{M}$ is an $n$-valent II-RBCM on ${\rm PSL}(2,p)$ with $p>5$, then $n\mid p(p^{2}-1)$.
\begin{enumerate}
  \item[\rm(i)] If $n=p$, then $\mathcal{M}\cong\mathcal{CM}_{\varpi(z)}^{\beta}$ for a unique $z\neq 0$.
  \item[\rm(ii)] If $n\mid p-1$, then $\mathcal{M}\cong\mathcal{CM}_{\omega({\rm ii},a)}^{\alpha^{(p-1)u/n}}$ for a unique pair $(a,u)$ such that
    \begin{itemize}
      \item $a\neq 0$ and $a^{2}\neq -1$;
      \item $(u,n)=1$ and $1\leq u<n/2$;
      \item $3a^{2}\neq -1,-2$ and $9a^{4}+9a^{2}+1\neq 0$ if $n=3$;
      \item $5a^{4}+5a^{2}+1\neq 0$ if $n=5$.
    \end{itemize}
  \item[\rm(iii)] If $n\mid p+1$, then $\mathcal{M}\cong\mathcal{CM}_{\tilde{\omega}({\rm ii},\check{a})}^{\gamma^{(p+1)v/n}}$ for a unique
       pair $(\check{a},v)$ such that
    \begin{itemize}
      \item $\check{a}\neq 0$ and $\check{a}^{2}\neq -e$;
      \item $(v,n)=1$ and $1\leq v<n/2$;
      \item $3\check{a}^{2}\neq -e,-2e$ and $9\check{a}^{4}+9\check{a}^{2}e+e^{2}\neq 0$ if $n=3$;
      \item $5\check{a}^{4}+5\check{a}e+e^{2}\neq 0$ if $n=5$. 
    \end{itemize}
\end{enumerate}
\end{thm}

\vspace{1cm}

{\bf Acknowledgements:}

I shall express my gratitude to Fangli Zhang for helpful suggestions on simplifying the proves of Lemma 3.2 and 4.3.

\end{document}